\documentclass{amsart}

\usepackage{t1enc}
\usepackage[latin2]{inputenc}
\usepackage{amsmath}
\usepackage{amssymb}
\usepackage{amsthm}
\usepackage{eufrak}
\usepackage{verbatim}
\usepackage{color}
\usepackage{pstricks,pst-node}
\usepackage{enumerate}
\usepackage{xmpmulti}
\usepackage{psfrag}
\usepackage{graphicx}

\newtheorem{thm}{Theorem}[section]
\newtheorem{prop}[thm]{Proposition}

\newtheorem{dfn}[thm]{Definition}
\newtheorem{clm}[thm]{Claim}
\newtheorem{lemma}[thm]{Lemma}

\newtheorem{cor}[thm]{Corollary}

\newcommand{\omg}{\omega_1}

\newcommand{\mad}{\mathcal{M}}
\newcommand{\slm}{S^\lambda_\mu}
\newcommand{\ter}{X[\lambda,\mu,\mad,\underline{C}]}

\author[D. Soukup]{D\'aniel Soukup}
\address{E\"otv\"os L\'or\'and University}
\email{daniel.t.soukup@gmail.com}
\title{Guessing clubs for \lowercase{a}D, non D-spaces}

\begin{document}
\maketitle

\begin{abstract} We prove that there exists a 0-dimensional, scattered $T_2$ space $X$ such that $X$ is aD but not linearly D, answering a question of Arhangel'skii. The constructions are based on Shelah's club guessing principles.
\end{abstract}

\section{Introduction}

The notion of a \emph{$D$-space} was probably first introduced by van Douwen and since than, many work had been done in this topic. Investigating the properties of $D$-spaces and the connections between other covering properties led to the definition of \emph{aD-spaces}, defined by Arhangel'skii in \cite{arhadd}. As it turned out, property aD is much more docile then property D. In \cite{arhcov} Arhangel'skii asked the following:\\

\textbf{Problem 4.6.} Is there a Tychonoff aD-space which is not a D-space?\\

A negative answer to this question would settle almost all of the questions about the relationship of classical covering properties to property D. Quite similarly, Guo and Junnila in \cite{linD} asked the following about a weakening of property D:\\

\textbf{Problem 2.12.} Is every aD-space linearly D?\\

In G. Gruenhage's survey on D-spaces \cite{gg}, another version of this question is stated (besides the original Arhangel'skii), namely:\\

\textbf{Question 3.6(2)} Is every scattered, aD-space a D-space?\\

The main result of this paper is the following answer to the questions above.

\begin{thm}\label{main} There exists a 0-dimensional $T_2$ space $X$ such that $X$ is scattered, aD and non linearly D.
\end{thm}

In \cite{dsoukup} the author showed that the existence of a locally countable, locally compact space $X$ of size $\omg$ which is aD and non linearly D is independent of ZFC. Here we refine those methods and using Shelah's club guessing theory we answer the above questions in ZFC.

The paper has the following structure. Sections \ref{defsec}, \ref{madsec} and \ref{guesssec} gather all the necessary facts about D-spaces, MAD families and club guessing. In Section \ref{consec} we define spaces $\ter$, where $\lambda$ and $\mu=cf(\mu)$ are cardinals, $\mad$ is a MAD family on $\mu$ and $\underline{C}$ is a guessing sequence. It is shown in Claim \ref{basic} that
 \begin{enumerate}[(0)]
 \item $\ter$ is always $T_2$, 0-dimensional and scattered.
  \end{enumerate}
  Section \ref{Dsec} contains two important results:
\begin{enumerate}[(1)]
  \item $\ter$ is not linearly D if $cf(\lambda)\geq\mu$ (see Corollary \ref{nlind}),
  \item $\ter$ is aD under certain assumptions (see Corollary \ref{fokov}).
\end{enumerate}
Finally in Section \ref{examp} we show how to produce such spaces $\ter$ depending on the cardinal arithmetic and using Shelah's club guessing.

The author would like to thank Assaf Rinot for his ideas and advices to look deeper into the theory of club guessing in  ZFC.

\section{Definitions}\label{defsec}

An \emph{open neighborhood assignment} (ONA, in short) on a space $(X,\tau)$ is a map $U:X\rightarrow\tau$ such that $x\in U(x)$ for every $x\in X$. X is said to be a \emph{D-space} if for every neighborhood assignment $U$, one can find a closed discrete $D\subseteq X$ such that $X=\bigcup_{d\in D}U(d)=\bigcup U[D]$ (such a set $D$ is called a \emph{kernel for $U$}). In \cite{arhadd} the authors introduced property \emph{aD}:

\begin{dfn}A space $(X,\tau)$ is said to be \emph{aD} iff for each closed $F\subseteq X$ and for each open cover $\mathcal{U}$ of $X$ there is a closed discrete $A\subseteq F$ and $\phi:A\rightarrow \mathcal{U}$ with $a\in\phi(a)$ for all $a\in A$ such that $F\subseteq \cup\phi[A]$.
\end{dfn}

It is clear that D-spaces are aD. Proving that a space is aD, the notion of an \emph{irreducible space} will play a key role. A space $X$ is \emph{irreducible} iff every open cover  $\mathcal{U}$ has a \emph{minimal open refinement} $\mathcal{U}_0$; meaning that no proper subfamily of $\mathcal{U}_0$ covers $X$. In \cite{arhcov} Arhangel'skii showed the following equivalence.

\begin{thm}[{\cite[Theorem 1.8]{arhcov}}]A $T_1$-space $X$ is an aD-space if and only if every closed subspace of $X$ is irreducible.
\end{thm}

Another generalization of property D is due to Guo and Junnila \cite{linD}. For a space $X$ a cover $\mathcal{U}$ is \emph{monotone} iff it is linearly ordered by inclusion.

\begin{dfn}A space $(X,\tau)$ is said to be \emph{linearly D} iff for any ONA $U:X\rightarrow\tau$ for which $\{U(x):x\in X\}$ is monotone, one can find a closed discrete set $D\subseteq X$ such that $X=\bigcup U[D]$.
\end{dfn}

We will use the following characterization of linear D property. A set $D\subseteq X$ is said to be \emph{$\mathcal{U}$-big} for a cover  $\mathcal{U}$ iff there is no $U\in \mathcal{U}$ such that $D\subseteq U$.

\begin{thm}[{\cite[Theorem 2.2]{linD}}] \label{linD} The following are equivalent for a $T_1$-space X:
\begin{enumerate}
  \item X is linearly D.
  \item For every non-trivial monotone open cover $\mathcal{U}$ of $X$, there exists a closed discrete $\mathcal{U}$-big set in $X$.
\end{enumerate}
\end{thm}

\section{Notes on MAD families}\label{madsec}

As MAD families will play an essential part in our constructions we observe some easy facts about them. Let $\mu$ be any infinite cardinal. We call $\mad\subseteq [\mu]^\mu$ an \emph{almost disjoint family} if $|M\cap N|<\mu$ for all distinct $M,N\in\mad$. $\mad$ is a \emph{maximal almost disjoint family} (in short, a \emph{MAD family}) if for all $A\in[\mu]^\mu$ there is some $M\in\mad$ such that $|A\cap M|=\mu$. % We will be interested in cases where $\kappa=\omega$ or $\omg$.

We will use the following rather trivial combinatorial fact.

\begin{clm}\label{madmetsz}
 Let $\mad\subseteq [\mu]^\mu$ be a MAD family and $\mad=\{M^\varphi:\varphi<\kappa\}$. Suppose that $N\in [\mu]^\mu$ and $|N\setminus \cup\mad'|=\mu$ for all $\mad'\in[\mad]^{<\mu}$. Then $|\Phi|>\mu$ for $\Phi=\{\varphi<\kappa: |N\cap M^\varphi|=\mu\}$.
\end{clm}
\begin{proof}
 If $|\Phi|<\mu$ then with $\widetilde{N}=N\setminus\bigcup\{M^\varphi:\varphi\in\Phi\}\in [\mu]^\mu$ we can extend the MAD family, which is a contradiction. If $|\Phi|= \mu$ then let $\Phi=\{\varphi_\zeta:\zeta<\mu\}$. By transfinite induction, construct $\widetilde{N}=\{n_\xi:\xi<\mu\}$ such that $n_\xi\in N\setminus(\bigcup\{M^{\varphi_\zeta}:\zeta<\xi\}\cup\{n_\zeta:\zeta<\xi\})$ for $\xi<\mu$. It is straightforward that $\widetilde{N}\notin \mad$ and $\mad\cup\{\widetilde{N}\}$ is almost disjoint, which is a contradiction.
\end{proof}

From our point of view the sizes of MAD families are important. Clearly there is a MAD family on $\omega$ of size $2^\omega$. The analogue of this does not always hold for $\omg$. Baumgartner in \cite{baum} proves that it is consistent with ZFC that there is no almost disjoint family on $\omg$ of size $2^{\omg}$. However, we have the following fact.

\begin{clm}\label{madsize} If $2^\omega=\omg$ then there is a MAD family $\mad$ on $\omg$ of size $2^{\omg}$.
\end{clm}

In Section \ref{examp} we use \emph{nonstationary MAD families} $\mad_{NS}\subseteq [\mu]^\mu$ meaning that $\mad_{NS}$ is a MAD family such that every $M\in\mad_{NS}$ is nonstationary in $\mu$. Observe, that using Zorn's lemma to almost disjoint families of nonstationary sets of $\mu$ we can get nonstationary MAD families.

\section{Fragments of Shelah's club guessing}\label{guesssec}

The constructions of the upcoming sections will use the following amazing results of Shelah. For a cardinal $\lambda$ and a regular cardinal $\mu$ let $S^\lambda_\mu$ denote the ordinals in $\lambda$ with cofinality $\mu$. For an $S\subseteq S^\lambda_\mu$ an \emph{S-club sequence} is a sequence $\underline{C}=\langle C_\delta:\delta\in S\rangle$ such that $C_\delta\subseteq\delta$ is a club in $\delta$ of order type $\mu$.

\begin{thm}[{\cite[Claim 2.3]{shelah1}}]\label{shelah1}Let $\lambda$ be a cardinal such that $cf(\lambda)\geq \mu^{++}$ for some regular $\mu$ and let $S\subseteq S_\mu^\lambda$ stationary. Then there is an $S$-club sequence $\underline{C}=\langle C_\delta:\delta\in S\rangle$ such that for every club $E\subseteq \lambda$ there is $\delta\in S$ (equivalently, stationary many) such that $C_\delta\subseteq E$.
\end{thm}

A detailed proof of Theorem \ref{shelah1} can be found in \cite[Theorem 2.17]{card}.

\begin{thm}[{\cite[Claim 3.5]{shelah2}}] \label{shelah2} Let $\lambda$ be a cardinal such that $\lambda=\mu^+$ for some uncountable, regular $\mu$ and $S\subseteq S^\lambda_\mu$ stationary. Then there is an $S$-club sequence $\underline{C}=\langle C_\delta:\delta\in S \rangle$ such that $C_\delta=\{\alpha^\delta_\zeta:\zeta<\mu\}\subseteq \delta$ and for every club $E\subseteq \lambda$ there is $\delta\in S$ (equivalently, stationary many) such that: $$\{\zeta<\mu:\alpha^\delta_{\zeta+1}\in E\} \text{ is stationary}.$$
\end{thm}

For a detailed proof, see \cite{dummies}.

\section{The general construction}\label{consec}

\begin{dfn}Let $\lambda>\mu=cf(\mu)$ be infinite cardinals. Let $\mad\subseteq [\mu]^\mu$ be a MAD family, $\mad=\{M^\varphi:\varphi<\kappa\}$ and let $\underline{C}=\{C_\alpha:\alpha\in S^\lambda_\mu\}$ denote an $\slm$-club sequence. We define a topological space $X=X[\lambda,\mu,\mad,\underline{C}]$ as follows. The underlying set of our topology will be a subset of the product $\lambda\times \kappa$. Let
\begin{itemize}
  \item $X_\alpha=\{(\alpha,0)\}$ for $\alpha\in \lambda\setminus S^\lambda_\mu$,
  \item $X_\alpha=\{\alpha\}\times\kappa$ for $\alpha\in S^\lambda_\mu$,
  \item $X=\bigcup \{X_\alpha:\alpha<\lambda\}$.
\end{itemize}
Let $C_\alpha=\{a_\alpha^\xi:\xi<\mu\}$ denote the increasing enumeration for $\alpha\in\slm$. For each $\alpha\in\slm$ let
\begin{itemize}
  \item $I_\alpha^\xi=(a_\alpha^\xi,a_\alpha^{\xi+1}]$ for $\xi\in \text{succ} (\mu)\cup\{0\}$,
  \item $I_\alpha^\xi=[a_\alpha^\xi,a_\alpha^{\xi+1}]$ for $\xi\in\lim (\mu)$.
\end{itemize}
Note that $\bigcup\{I_\alpha^\xi:\xi<\mu\}=(a_\alpha^0, \alpha)$ is a disjoint union.

Define the topology on $X$ by neighborhood bases as follows;
\begin{enumerate}[(i)]
  \item for $\alpha\in \slm$ and $\varphi<\kappa$ let $$U((\alpha,\varphi),\eta)=\{(\alpha,\varphi)\}\cup\bigcup\{X_\gamma:\gamma\in\cup\{I_\alpha^\xi:\xi\in M^\varphi\setminus \eta\}\}\text{ for }  \eta<\mu$$
      and let $$B(\alpha,\varphi)=\{U((\alpha,\varphi),\eta):\eta<\mu\}$$ be a base for the point $(\alpha,\varphi)$.

      \psfrag{a}{$\alpha$}
\psfrag{ab}{$(\alpha,\varphi)$}
\psfrag{om}{$\lambda$}
\psfrag{Xa}{$X_\alpha$}
\psfrag{an}{\tiny{$a_\alpha^\xi$}}
\psfrag{an1}{\tiny{$a_\alpha^{\xi+1}$}}
\psfrag{I}{$\empty$}
\begin{center}
    \includegraphics[keepaspectratio, width=5 cm]{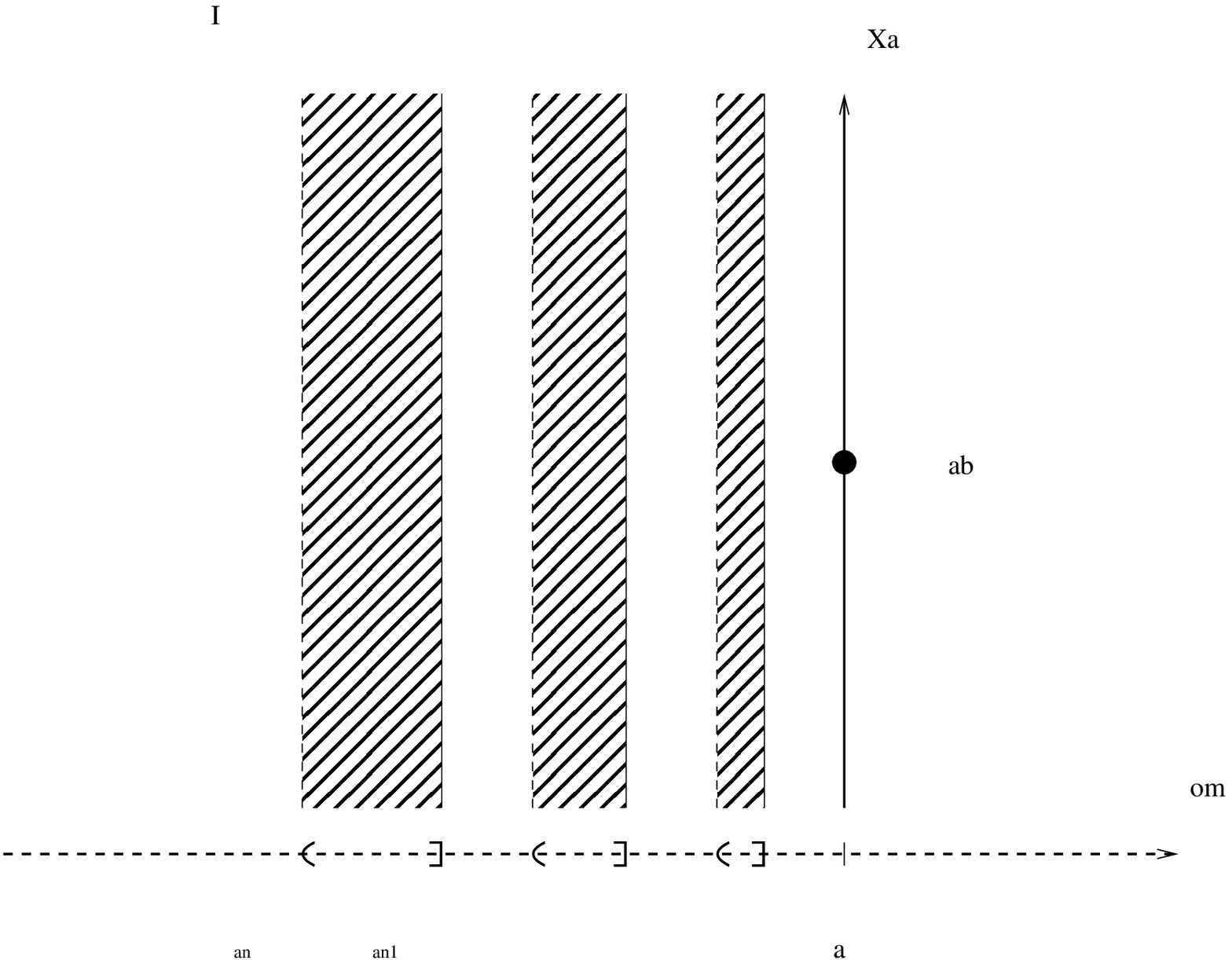}
\end{center}

  \item for $\alpha\in S^\lambda_{<\mu}\cup \text{succ}(\lambda)\cup \{0\}$ let $(\alpha,0)$ be an isolated point,
  \item for $\alpha\in S^\lambda_{\mu'}$ where $\mu'>\mu$ let $$U(\alpha,\beta)=\bigcup\{X_\gamma:\beta<\gamma\leq\alpha\}\text{ for }  \beta<\alpha$$
      and let $$B(\alpha)=\{U(\alpha,\beta):\beta<\alpha\}$$ be a base for the point $(\alpha,0)$.
\end{enumerate}
\end{dfn}

It is straightforward to check that these \emph{basic open sets} form neighborhood bases.

 \begin{center}
    $\star$
 \end{center}

  Fix some cardinals $\lambda>\mu=cf(\mu)$, a MAD family $\mad=\{M^\varphi:\varphi<\kappa\}\subseteq[\mu]^\mu$ and $\slm$-club sequence $\underline{C}$. In the following $X=\ter$.

\begin{clm}  \label{basic}The space $\ter$ is 0-dimensional, $T_2$ and scattered. Observe that
\begin{enumerate}[(a)]
  \item $X_\alpha$ is closed discrete for all $\alpha<\lambda$, moreover
  \item $\bigcup\{X_\alpha:\alpha\in A\}$ is closed discrete for all $A\in[\lambda]^{<\mu}$,
  \item $X_{\leq\alpha}=\bigcup\{X_\beta:\beta\leq\alpha\}$ is clopen for all $\alpha<\lambda$.
\end{enumerate}
\end{clm}
\begin{proof}First we prove that $\ter$ is $T_2$. Note that
\begin{enumerate}[($*$)]
        \item $\bigcup\{X_\gamma:\beta<\gamma\leq\alpha\}$ is clopen for all $\beta<\alpha<\lambda$.
\end{enumerate}
 Thus $(\alpha,\varphi),(\alpha',\varphi')\in X$ can be separated trivially if $\alpha\neq\alpha'$. Suppose that $\alpha=\alpha'\in\slm$ and $\varphi\neq\varphi'<\kappa$. There is $\eta<\mu$ such that $(M^\varphi\cap M^{\varphi'})\setminus \eta=\emptyset$ since $|M^\varphi\cap M^{\varphi'}|<\mu$. Thus $U((\alpha,\varphi),\eta)\cap U((\alpha,\varphi'),\eta)=\emptyset$.

 Next we show that $\ter$ is 0-dimensional. By $(*)$ it is enough to prove that $U((\alpha,\varphi),\eta)$ is closed for all $\alpha\in \slm$, $\varphi<\kappa$ and $\eta<\mu$. Suppose $x=(\alpha',\varphi')\in X\setminus U((\alpha,\varphi),\eta)$, we want to separate $x$ from $U((\alpha,\varphi),\eta)$ by an open set. Let $\alpha=\alpha'$. There is $\eta'<\mu$ such that $(M^\varphi\cap M^{\varphi'})\setminus \eta'=\emptyset$, thus $U((\alpha,\varphi),\eta)\cap U((\alpha,\varphi'),\eta')=\emptyset$. Let $\alpha\neq \alpha'$. If $\alpha'\in S^\lambda_{<\mu}\cup \text{succ}(\lambda)\cup \{0\}$ then $x$ is isolated, thus we are done. Suppose $\alpha\in S^\lambda_{\mu'}$ where $\mu'\geq\mu$. Then $\beta=\sup (C_\alpha\setminus \alpha')<\alpha'$ thus $U(\alpha',\beta)\cap U((\alpha,\varphi),\eta)=\emptyset$.

  $\ter$ is scattered since $\ter$ is right separated by the lexicographical ordering on $\lambda\times\kappa$.

   (a) and (c) is trivial, we prove (b). Suppose $x=(\alpha',\varphi')\in X$, we prove that there is a neighborhood $U$ of $x$ such that $|U\cap \bigcup\{X_\alpha:\alpha\in A\}|\leq 1$. If $\alpha'\in S^\lambda_{<\mu}\cup \text{succ}(\lambda)\cup \{0\}$ then $x$ is isolated, thus we are done. Suppose $\alpha\in S^\lambda_{\mu'}$ where $\mu'\geq\mu$. Then $\beta=\sup (A\setminus \alpha')<\alpha'$ thus the open set $U=\{x\}\cup\bigcup\{X_\gamma:\beta<\gamma<\alpha\}$ will do the job.
\end{proof}

\section{Focusing on property D and \lowercase{a}D}\label{Dsec}

 Again fix some cardinals $\lambda>\mu=cf(\mu)$, a MAD family $\mad=\{M^\varphi:\varphi<\kappa\}\subseteq[\mu]^\mu$ and $\slm$-club sequence $\underline{C}$. Our next aim is to investigate the spaces $X=\ter$ concerning property D and aD.

\begin{dfn} Let $\pi(F)=\{\alpha<\lambda:F\cap X_\alpha\neq\emptyset\}$ for $F\subseteq X$. $F$ is said to be \emph{(un)bounded} if $\pi(F)$ is (un)bounded in $\lambda$.
\end{dfn}

\begin{clm}\label{sorkomp}If $F\subseteq X$ and $\pi(F)$ accumulates to $\alpha\in S^\lambda_\eta$ such that $\mu\leq\eta<\lambda$ then $F'\cap X_\alpha\neq\emptyset$.
\end{clm}
\begin{proof} If $\eta>\mu$ then $X_\alpha=\{(\alpha,0)\}$ and each neighborhood $U(\alpha,\beta)$ of $(\alpha,0)$ intersects $F$. Thus $F'\cap X_\alpha\neq\emptyset$. Let us suppose that $\pi(F)$ accumulates to $\alpha\in S^\lambda_\mu$. Since $\bigcup\{I_\alpha^\xi:\xi<\mu\}=(a_\alpha^0,\alpha)$, the set $N=\{\xi<\mu:I_\alpha^\xi\cap \pi(F)\neq\emptyset\}$ has cardinality $\mu$. Thus there is some $\varphi<\kappa$ such that $|N\cap M^\varphi|=\mu$, since $\mad$ is MAD family. It is straightforward that $(\alpha,\varphi)\in F'$ since $U((\alpha,\varphi),\eta)\cap F\neq\emptyset$ for all $\eta<\mu$.
\end{proof}

\begin{cor} \label{nlind} If $cf(\lambda)\geq \mu$ then a closed unbounded subspace $F\subseteq X$ is not a linearly D-subspace of $X$. Hence $\ter$ is not a linearly D-space.
\end{cor}
\begin{proof} Let $F\subseteq X$ be closed unbounded. $|\pi(D)|<\mu$ for every closed discrete $D\subseteq X$ by Claim \ref{sorkomp}. Thus there is no big closed discrete set for the open cover $\{X_{\leq \alpha}:\alpha<\lambda\}$ which shows that $F$ is not linearly D by Theorem \ref{linD}.
\end{proof}

Our aim now is to prove that in certain cases the space $\ter$ is an aD-space, equivalently every closed subspace of it is irreducible.

\begin{clm}\label{korld} Every closed, bounded subspace $F\subseteq X$ is a D-subspace of $X$; hence $F$ is irreducible.
\end{clm}
\begin{proof} We prove that $F\subseteq X$ is a D-subspace of $X$ by induction on $\alpha=\sup\pi(F)<\lambda$. Let $U: F\rightarrow \tau$ be an ONA. If $\alpha$ is a successor (or $\alpha=0$), then $F_0=F\setminus U((\alpha,0))$ is closed and $\sup(F_0)<\alpha$ thus we are easily done by induction.

Let $\alpha\in S^\lambda_{\mu'}$ where $\mu\leq\mu'<\lambda$. Then $\sup\pi(F_0)<\alpha$  where $F_0=F\setminus \cup U[X_\alpha\cap F]$ by Claim \ref{sorkomp}. Thus we are easily done by induction and the fact that $X_\alpha$ is closed discrete.

Now let $\nu=cf(\alpha)<\mu$, let $\sup\{\alpha_\xi:\xi<\nu\}=\alpha$ such that $\alpha_0=0$ and $\{\alpha_\xi:\xi<\nu\}$ is strictly increasing. Let $J_\xi=\bigcup\{X_\gamma:\alpha_\xi\leq\gamma\leq\alpha_{\xi+1}\}$ if $\xi<\nu$ is limit or $\xi=0$ and $J_\xi=\bigcup\{X_\gamma:\alpha_\xi<\gamma\leq\alpha_{\xi+1}\}$ if $\xi<\nu$ is a successor. Let $J_\nu=X_\alpha$. Clearly $\{J_\xi:\xi\leq\nu\}$ is a discrete family of disjoint clopen sets such that $\bigcup \{J_\xi:\xi\leq\nu\}=X_{\leq\alpha}$. $F=\bigcup\{F^\xi:\xi\leq\nu\}$ where $F^\xi=F\cap J_\xi$ is closed for $\xi\leq\nu$. By induction, for all $\xi<\nu$ there is some closed discrete kernel $D^\xi\subseteq F^\xi$ for the restriction of $U$ to $F^\xi$. Let $D^\nu=F^\nu$. Then $D=\bigcup\{D^\xi:\xi\leq\nu\}$ is closed discrete and $F\subseteq\cup U[D]$.
\end{proof}

To handle the unbounded closed subsets we need the following definition.

\begin{dfn} Let $F_\alpha=F\cap X_\alpha$  for $F\subseteq X$ and $\alpha<\lambda$. A  \emph{subset $F\subseteq X$ is high enough} if $$|\{\alpha<\lambda: |F_\alpha|=|F|\}|\geq \mu.$$ We say that a \emph{subset $F\subseteq X$ is high} if every closed unbounded subset of $F$ is high enough.
\end{dfn}

The following rather technical claim will be useful.

 \begin{clm}\label{tech}For any $F\subseteq X$ and ONA $U:F\rightarrow \tau$ such that $U(x)$ is a basic open neighborhood of $x\in F$, let $$Y_F=\{x\in F: \exists\alpha<\lambda:F_\alpha\subseteq U(x),|F_\alpha|=|F|\},$$ $$\Gamma_F=\{\alpha<\lambda: |F_\alpha|=|F|,\exists x\in F:  F_\alpha\subseteq U(x)\}.$$
If $F$ is closed and high enough then $Y_F,\Gamma_F\neq\emptyset$.
\end{clm}
\begin{proof}Since $Y_F\neq\emptyset$ iff $\Gamma_F\neq\emptyset$, it is enough to show that there is some $x\in Y_F$. Since $F$ is high enough, $|Z|\geq\mu$ for $Z=\{\alpha'<\lambda: |F|=|F_{\alpha'}|\}$. Let $D=\bigcup\{F_{\alpha'}:\alpha'\in Z\}\subseteq F$. Let $\beta\in\slm$ be an accumulation point of $Z=\pi(D)$. Then by Claim \ref{sorkomp} there is some $x\in D'\cap X_\beta$ thus $x\in F$. Clearly $x\in Y_F$.
\end{proof}

\begin{thm} \label{fotetel}If the closed unbounded $F\subseteq X$ is high then $F$ is irreducible.
\end{thm}
\begin{proof}Suppose that $\mathcal{U}$ is an open cover of $F$. We can suppose that we refined it to the form $\{U(x):x\in F\}$ where each $U(x)$ is basic open. From Claim \ref{tech} we know that $Y_F,\Gamma_F\neq\emptyset$. We define $Y^\xi\subseteq F$ by induction.
  \begin{itemize}
  \item Let $\alpha_0\in\Gamma_F$ and $Y^0=\{x\in Y_F: F_{\alpha_0}\subseteq U(x)\}$. Fix some $h^0:Y^0\rightarrow F_{\alpha_0}$ injection; this exists because $|F_{\alpha_0}|=|F|\geq |Y_F|\geq|Y^0|$.
  \item Suppose we defined $\alpha_\zeta<\lambda$ and $Y^\zeta$ for $\zeta<\xi$. Let $$F^\xi=F\setminus\bigl(\bigcup\bigl\{U(x):x\in\cup\{Y^\zeta:\zeta<\xi\}\bigr\}\cup X_{\leq\alpha}\bigr)$$ where $\alpha=\sup \{\alpha_\zeta:\zeta<\xi\}$.
      \item If $F^\xi$ is bounded then stop. Notice that $F_\xi$ is bounded iff $F\setminus\bigcup\bigl\{U(x):x\in\cup\{Y^\zeta:\zeta<\xi\}\bigr\}$ is bounded.
      \item Suppose $F^\xi$ is unbounded. $F^\xi\subseteq F$ is closed either thus $F^\xi$ is high enough since $F$ is high. Hence $Y_{F^\xi},\Gamma_{F^\xi}
          \neq\emptyset$.
     \item Let $\alpha_\xi\in\Gamma_{F^\xi}$; thus $|F^\xi_{\alpha_\xi}|=|F^\xi|$ and $F^\xi_{\alpha_\xi}$ is covered by some $U(x)$ for $x\in F^\xi$. Let $Y^\xi=\{x\in Y_{F^\xi}: F^\xi_{\alpha_\xi}\subseteq U(x)\}$. Fix some $h^\xi:Y^\xi\rightarrow F^\xi_{\alpha_\xi}$ injection; this exists because $|F^\xi_{\alpha_\xi}|=|F^\xi|\geq |Y_{F^\xi}|\geq|Y^\xi|$.

\end{itemize}
\begin{lemma}The induction stops before $\mu$ many steps.
\end{lemma}
\begin{proof} Suppose we defined this way $\{\alpha_{\xi}:\xi<\mu\}$ and let $\alpha=\sup\{\alpha_{\xi}:\xi<\mu\}\in\slm$. Let $D=\bigcup\{F_{\alpha_\xi}:\xi<\mu\}$. By Claim \ref{sorkomp} there is some $x\in D'\cap X_\alpha$, thus $x\in F$ either. Clearly $F_{\alpha_\xi}\subseteq U(x)$ for $\mu$ many $\xi<\mu$. By the definition of the induction
 \begin{enumerate}[$(*)$]
 \item for every $\zeta<\xi<\mu$ and every $y\in Y^{\zeta}$: $F^\xi_{\alpha_\xi}\cap U(y)=\emptyset$
\end{enumerate}
Clearly by $(*)$, $x\notin Y^\zeta$ for all $\zeta<\mu$ since there is $\zeta<\xi<\mu$ such that $F^\xi_{\alpha_\xi}\subseteq U(x)$. Moreover $x\notin U(y)$ for every $y\in Y^\zeta$ and $\zeta<\mu$; if $x\in U(y)$ then since $x\neq y$ there is some $\beta<\alpha$ such that $\bigcup\{X_\gamma:\beta<\gamma\leq\alpha\}\subseteq U(y)$. This contradicts $(*)$ since there is $\zeta<\xi<\mu$ such that $\beta<\alpha_\xi$, thus $F^\xi_{\alpha_\xi}\subseteq U(y)$. Thus $x\in F^\xi$ for all $\xi<\mu$. Then $x\in Y^\xi$ for all $\xi<\mu$ such that $F_{\alpha_\xi}\subseteq U(x)$. This is a contradiction.
\end{proof}
Thus let us suppose that the induction stopped at step $\xi< \mu$, meaning that $\widetilde{F}=F\setminus\bigcup\{U(x):x\in Y\}$ is bounded where $Y=\cup\{Y^\zeta:\zeta<\xi\}$. Let $h=\bigcup\{h^\zeta:\zeta<\xi\}$, $h:Y\rightarrow F$ is a 1-1 function since the sets $\text{dom}(h^\zeta)=Y^\zeta$ and $ran(h^\zeta)\subseteq F^\zeta_{\alpha_\zeta}$ are pairwise disjoint for $\zeta<\xi$. Note that $\text{ran}(h)\subseteq \bigcup\{F_{\alpha_\zeta}:\zeta<\xi\}$ is closed discrete by Claim \ref{basic}.  For $x\in Y$ let $$U_0(x)=(U(x)\setminus \text{ran}(h))\cup\{h(x)\},$$ note that $U_0(x)$ is open. Then $$\bigcup\{U_0(x):x\in Y\}=\bigcup\{U(x):x\in Y\}$$
is a minimal open refinement, since $h(x)$ is only covered by $U_0(x)$ for all $x\in Y$. Let $\mathcal{U}_0=\{U_0(x):x\in Y\}$

 Let $V(x)=U(x)\setminus \bigcup\{F_{\alpha_\zeta}:\zeta<\xi\}$. Then $\mathcal{V}=\{V(x):x\in\widetilde{F}\}$ is an open cover of $\widetilde{F}$, refining $\mathcal{U}$; $F_{\alpha_\zeta}\cap \widetilde{F}=\emptyset$ by construction for all $\zeta<\xi$. $\widetilde{F}$ is closed and bounded thus irreducible by Claim \ref{korld}, hence there is an irreducible open refinement $\mathcal{V}_0$ of $\mathcal{V}$. It is straightforward that $\mathcal{V}_0\cup \mathcal{U}_0$ is a minimal open refinement of $\mathcal{U}$ covering $F$.
\end{proof}

\begin{cor}\label{fokov} Suppose that $\lambda>\mu=cf(\mu)$ are infinite cardinals such that $cf(\lambda)\geq \mu$. Let $\mad=\{M^\varphi:\varphi<\kappa\}\subseteq[\mu]^\mu$ be a MAD family and $\underline{C}$ an $\slm$-club sequence. If $\ter$ is high then $\ter$ is a 0-dimensional, Hausdorff, scattered space which is aD however not linearly D.
\end{cor}
\begin{proof}
$\ter$ is 0-dimensional, Hausdorff and scattered by Claim \ref{basic} and not linearly D by Corollary \ref{nlind}. It suffices to show that every closed $F\subseteq X$ is irreducible. If $F$ is bounded then $F$ is a D-space by Claim \ref{korld} hence irreducible. If $F$ is unbounded, then $F$ is high since $X$ is high. Hence $F$ is irreducible by Theorem \ref{fotetel}.
\end{proof}

\section{Examples of \lowercase{a}D, non linearly D-spaces}\label{examp}

In this section we give examples of aD, non linearly D-spaces of the form $X=\ter$. First let us make an observation.

\begin{clm}\label{toszlop} If $C_\alpha\subseteq \pi(F)'$ for a closed $F\subseteq X$ and $\alpha\in\slm$, then $F_\alpha=X_\alpha$.
\end{clm}
\begin{proof}Clearly $\bigcup\{X_\gamma:\gamma\in I_\alpha^\xi\}\cap F\neq\emptyset$ for all $\xi<\mu$. Thus every point in $X_\alpha$ is an accumulation point of $F$, thus $F_\alpha=X_\alpha$ since $F$ is closed.
\end{proof}

Corollaries \ref{pelda1} and \ref{pelda2} below give certain examples of high $\ter$ spaces.

\begin{prop}
 Suppose that $\mu$ is a regular cardinal, $cf(\lambda)\geq\mu^{++}$. Let $\underline{C}$ be an $\slm$-club guessing sequence from Theorem \ref{shelah1}. If $\mad\subseteq [\mu]^\mu$ is a MAD family of size at least $\lambda$ then $\ter$ is high.
\end{prop}
\begin{proof}
 Let $F\subseteq X$ closed, unbounded. Then $\pi(F)'$ is a club in $\lambda$, hence there exists a stationary $S\subseteq \slm$ such that $C_\alpha\subseteq \pi(F)'$ for all $\alpha\in S$. Thus $F_\alpha=X_\alpha$ by Claim \ref{toszlop} hence $|F_\alpha|=|\mad|=|X|$ for all $\alpha\in S$.
\end{proof}

\begin{cor}\label{pelda1}
 \begin{enumerate}
  \item Suppose that $2^\omega\geq\omega_2$. Let $\mad$ be a MAD family on $\omega$ of size $2^\omega$ and let $\underline{C}$ be an $S^{\omega_2}_\omega$-club guessing sequence from Theorem \ref{shelah1}. Then $X[\omega_2,\omega,\mad,\underline{C}]$ is high.
\item Suppose that $2^\omega=\omega_1$ and $2^{\omega_1}\geq \omega_3$.  Let $\mad$ be a MAD family on $\omega_1$ of size $2^{\omega_1}$ (exists by Claim \ref{madsize}) and let $\underline{C}$ be an $S^{\omega_3}_{\omega_1}$-club guessing sequence from Theorem \ref{shelah1}. Then $X[\omega_3,\omega_1,\mad,\underline{C}]$ is high.
 \end{enumerate}
\end{cor}

\begin{prop}
 Suppose that $\lambda=\mu^+ >\mu=cf(\mu)>\omega$ and let $\underline{C}$ be an $S^{\mu^+}_\mu$-club guessing sequence from Theorem \ref{shelah2}. If there is a nonstationary MAD family $\mad_{NS}\subseteq [\mu]^\mu$ such that $|\mad_{NS}|=\mu^+$ then $X=X[\mu^+,\mu,\mad_{NS},\underline{C}]$ is high.
\end{prop}
\begin{proof}
 Let $\mad_{NS}=\{M^\varphi:\varphi<\mu^+\}$ and $\underline{C}=\langle C_\alpha:\alpha\in  S^{\mu^+}_\mu \rangle$ such that $C_\alpha=\{a_\alpha^\xi:\xi<\mu\}\subseteq \alpha$. Suppose that the closed $F\subseteq X$ is unbounded. Then  $\pi(F)'$ is a club in $\mu^+$, hence there exists a stationary $S\subseteq S^{\mu^+}_\mu$ such that $$N_\alpha=\{\xi<\mu: a_\alpha^{\xi+1}\in\pi(F)'\} \text{ is stationary in $\mu$}$$
 for all $\alpha\in S$. Fix any $\alpha\in S$, we prove that $|F_\alpha|=|F|$. $N_\alpha$ is stationary so by applying Claim \ref{madmetsz} we get that $|\Phi_\alpha|=\mu^+$ for $\Phi_\alpha=\{\varphi<\mu^+:|N_\alpha\cap M^\varphi|=\mu\}$. Note that $F\cap \bigcup\{X_\gamma:\gamma\in I_\alpha^\xi\}\neq\emptyset$ for $\xi\in N_\alpha$. Thus $(\alpha,\varphi)$ is an accumulation point of $F$ for $\varphi\in\Phi_\alpha$, hence $\{\alpha\}\times \Phi_\alpha\subseteq F_\alpha$. Thus $|F_\alpha|=\mu^+=|X|$.
\end{proof}

\begin{cor}\label{pelda2}
 Suppose that $2^{\omega_1}=\omega_2$. Let $\underline{C}$ be an $S^{\omega_2}_{\omega_1}$-club guessing sequence from Theorem \ref{shelah2} and let $\mad_{NS}$ be a nonstationary MAD family on $\omega_1$. Then $X[\omega_2,\omega_1,\mad_{NS},\underline{C}]$ is high.
\end{cor}

Thus, by all means we can deduce the proof of Theorem \ref{main}.

\begin{proof}[Proof of Theorem \ref{main}] Note that in any model of ZFC, either $(2^\omega\geq\omega_2)$ or $(2^\omega=\omg\wedge 2^{\omg}\geq\omega_3)$ or $(2^{\omg}=\omega_2)$. Using Corollaries \ref{pelda1} and \ref{pelda2} above, depending on the sizes of $2^\omega$ and $2^{\omg}$, we see that there exists a high $\ter$ space. We are done by Corollary \ref{fokov}.
\end{proof}

\end{document}